\documentclass[11pt,english,a4paper]{article}

\pdfoutput=1

\usepackage[latin1]{inputenc}
\usepackage[T1]{fontenc}
\usepackage{lmodern}
\usepackage{babel}
\usepackage{amsmath,amssymb,amsthm}
\usepackage[totalwidth=400pt,totalheight=630pt]{geometry}
\usepackage{microtype}
\usepackage{hyperref}
\usepackage{url}

\newtheorem{theorem}{Theorem}[section]
\newtheorem{lemma}[theorem]{Lemma}
\newtheorem{proposition}[theorem]{Proposition}
\newtheorem{corollary}[theorem]{Corollary}

\theoremstyle{definition}

\newtheorem{example}[theorem]{Example}
\newtheorem{question}[theorem]{Question}

\theoremstyle{remark}

\numberwithin{equation}{section}

\newcommand{\Z}{\mathbb{Z}}
\newcommand{\N}{\mathbb{N}}
\newcommand{\T}{\mathbb{T}}

\newcommand{\C}{\mathbb{C}}

\DeclareMathOperator*{\aut}{Aut}

\def\K{\mathcal{K}}
\def\P{\mathcal{P}}

\def\ICCP{\mathcal{ICCP}}

\begin{document}

\title{On twisted group C$^*$-algebras associated with FC-hypercentral groups and other related groups}

\author{Erik B\'edos, Tron Omland \vspace{2.5em}}

\date{November 7, 2014}
\maketitle
\renewcommand{\sectionmark}[1]{}

\begin{abstract}
We show that the twisted group C$^*$-algebra associated with a discrete FC-hypercentral group is simple (resp.\ has a unique tracial state) if and only if Kleppner's condition is satisfied. This generalizes a result of J.~Packer for countable nilpotent groups. We also consider a larger class of groups, for which we can show that the corresponding reduced twisted group C$^*$-algebras have a unique tracial state if and only if Kleppner's condition holds.

\vspace{3.5em}


\vspace{0.5em}


\end{abstract}

\section{Introduction}

In this article, all groups will be considered as discrete groups. 
Letting $\sigma\colon G\times G\to \T$ denote a normalized two-cocycle on a group $G$, in other words, $\sigma \in Z^2(G, \T)$, we will say that $(G, \sigma)$ is \emph{C$^*$-simple} (resp.\ has \emph{the unique trace property}) whenever the reduced twisted  group C$^*$-algebra $C_r^*(G, \sigma)$ is simple (resp.\ has a unique tracial state). If $G$ is amenable, then we can equally consider the (full) twisted  group C$^*$-algebra $C^*(G, \sigma)$, since $C^*(G, \sigma)$ is canonically isomorphic to $C_r^*(G, \sigma)$ in this case (cf.\ \cite{ZM}). It is well known that a necessary condition for $(G, \sigma)$ to be C$^*$-simple (resp.\ have the unique trace property) is \emph{Kleppner's condition} \cite{Kle}, which says that every nontrivial $\sigma$-regular conjugacy class in $G$ is infinite.
In general, Kleppner's condition is not sufficient for $(G, \sigma)$ to be C$^*$-simple (resp.\ have the unique trace property). However, for certain classes of groups, Kleppner's condition is sufficient for both these properties to hold.
We will therefore say that  a group $G$ belongs to the class $\K_{C^*S}$ (resp. $\K_{UT}$) if, for every $\sigma \in Z^2(G, \T)$, we have that $(G, \sigma)$ is C$^*$-simple (resp.\ has the unique trace property) if and only if  $(G, \sigma)$ satisfies Kleppner's condition.  
Moreover, we will let $\K$ denote the intersection of $\K_{C^*S} $ and $\K_{UT}$, while $\K^{am}$ will denote the subclass of $\K$ consisting of amenable groups. 

\medskip It is a classical  fact that the class of finite groups is contained in $\K^{am}$; see for example \cite{Isa} or \cite{Kle}. J.~Packer \cite{Pa} has shown that all countable nilpotent groups belong to $\K^{am}$ (see also \cite{Sla} for the case of abelian groups).
Large families of non-amenable groups in $\K$ are described in \cite{Bed, Bed2, BC2}. We mention explicitly the class $\P$ introduced in \cite{BC2}, as we will refer to it later: it consists of all PH groups \cite{Pro} and of all groups with property ($P_{\rm com}$) \cite{BCH}. In particular, all weak Powers groups \cite{H} belong to $\P$, and the class $\P$ contains many amalgamated free products, HNN-extensions, hyperbolic groups, Coxeter groups, and lattices in semisimple Lie groups. Any group belonging to $\P$ is ICC (meaning that all of its nontrivial conjugacy classes are infinite), so Kleppner's condition is trivially satisfied for any $\sigma \in Z^2(G,\T)$ when $G$ lies in $\P$.

\medskip Our first aim in this paper is to show that the class $\K^{am}$ contains all \emph{FC-hypercentral} groups (cf.\ Theorem~\ref{K}). The class of FC-hypercentral groups \cite{Rob} is quite large: it contains for instance all groups that have only finite conjugacy classes (usually called FC-groups); moreover, it contains all virtually nilpotent groups and all FC-nilpotent groups. A simple way to describe that a group is FC-hypercentral is to say that it has no nontrivial ICC quotient group. We will recall the equivalent definitions in Section~\ref{FCH-section}. We just mention here that if a group $G$ is finitely generated, then $G$ is FC-hypercentral if and only if $G$ is virtually nilpotent, if and only if $G$ has polynomial growth. An interesting open problem raised by our result is whether the class $\K^{am}$ coincides with the class of FC-hypercentral groups.

\medskip In our proof of Theorem~\ref{K}, we will use the observation that if a group $G$ is amenable and $(G, \sigma)$ has the unique trace property for some  $\sigma \in Z^2(G,\T)$, then $(G, \sigma)$ is C$^*$-simple. This follows easily from the fact that $C_r^*(G, \sigma)$ has the QTS  property introduced by G.~Murphy \cite{Mur} whenever $G$ is amenable. Thus, the main burden of our proof will be to show that for an FC-hypercentral group, the unique trace property follows from Kleppner's condition. This will be achieved by streamlining and generalizing the proof of the same implication given by J.~Packer \cite{Pa} in the case of a countable nilpotent group. For completeness, we will also give another proof that for an FC-hypercentral group, the C$^*$-simplicity can be deduced from Kleppner's condition, by making use of a deep result of  S.~Echterhoff in \cite{Ech}. 
 
\medskip A consequence of Theorem~\ref{K} is that if $G$ is a countable FC-hypercentral group,  $\sigma \in Z^2(G, \T)$ and $(G,\sigma)$ satisfies Kleppner's condition, then $C^*(G, \sigma) \simeq C_r^*(G, \sigma)$ belongs to the class of separable simple nuclear C$^*$-algebras with a unique tracial state, a class that is of particular interest in the classification program for C$^*$-algebras.

\medskip In the second part of this paper, we consider a larger class of groups and show that it is contained in $\K_{UT}$. To describe this class, we first recall \cite{Rob} that any group $G$ has a canonical normal FC-hypercentral subgroup, $FCH(G)$, called its \emph{FC-hypercenter}. The quotient group $G/FCH(G)$ is an ICC group, that we will denote by $ICC(G)$. We will say that $G$ belongs to the class $\ICCP$ when $ICC(G)$ belongs to the class $\P$ mentioned above. As the trivial group is the only amenable group belonging to $\P$, we have that the class of FC-hypercentral  groups is the intersection of  $\ICCP$ with the class of amenable groups. Our result is that the class $\ICCP$ is contained in $\K_{UT}$ (cf.\ Theorem~\ref{FCH-P}). We believe that $\ICCP$ is also contained in $\K_{C^*S}$, and include a result supporting this conjecture.

\section{Preliminaries}

\subsection{On twisted group \texorpdfstring{C$^*$}{C*}-algebras}

Let $G$ denote a group with identity $e$ and let $ \sigma\colon G\times G \to \T$ denote a normalized two-cocycle
(sometimes called a multiplier) on $G$ with values in the circle group $\T$. We recall that $\sigma$ satisfies 
\begin{gather*}
\sigma(g,h)\sigma(gh,k)=\sigma(h,k)\sigma(g,hk)\, , \\
\sigma(g,e)=\sigma(e,g)=1
\end{gather*}
for all $g,h,k \in G$.

The set of all such two-cocycles will be denoted by $Z^2(G,\T)$, as in \cite{ZM}. The trivial two-cocycle is simply written as 1.
We will use the convention that when $\sigma =1$, we just drop $\sigma$ from all our notation. 

\medskip The \emph{left regular $\sigma$-projective unitary representation $\lambda_{\sigma}$ of $G$ on $B(\ell^2(G))$} is given by
\[
(\lambda_{\sigma}(g)\xi)(h)=\sigma(g,g^{-1}h)\, \xi(g^{-1}h)
\]
for $\xi \in \ell^2(G)$ and $g,h \in G$. Note that
\[
\lambda_{\sigma}(g)\, \delta_h=\sigma(g,h)\, \delta_{gh}
\]
for all $g,h\in G$ (where $\delta_g(h)=1$ if $g=h$ and $\delta_g(h)=0$ otherwise).
The \emph{reduced twisted group C$^*$-algebra} and the \emph{twisted group von~Neumann algebra of $(G,\sigma)$},
$C^*_r(G,\sigma)$ and $W^*(G,\sigma)$ are, respectively, the C$^*$-algebra and the von~Neumann algebra generated by $\lambda_{\sigma}(G)$.
The \emph{(full) twisted group C$^*$-algebra of $(G,\sigma)$}, $C^*(G,\sigma)$,
is the enveloping C$^*$-algebra of the Banach $^*$-algebra $\ell^1(G, \sigma)$,
equipped with the twisted convolution and involution (see \cite{ZM}).

\medskip The canonical tracial state on   $C_r^*(G,\sigma)$ will be denoted by $\tau$;
it is simply given as the restriction to $C_r^*(G,\sigma)$ of the vector state associated with $\delta_e$.
As is well-known, $\tau$ is faithful and satisfies $\tau(\lambda_\sigma(g))=0$ for every $g\neq e$.

\medskip We recall \cite{Kle, Pa, Om} that $g \in G$  is called \emph{$\sigma$-regular} if 
\[
\sigma(g,h)=\sigma(h,g) \text{ for every } h \in G \text{ that commutes with } g\, .
\]
If $g$ is $\sigma$-regular, then $hgh^{-1}$ is $\sigma$-regular for all $h$ in $G$,
so the notion of $\sigma$-regularity makes sense for conjugacy classes in $G$. 

As mentioned in the Introduction,
the pair $(G,\sigma)$ will be said to satisfy \emph{Kleppner's condition} if every nontrivial $\sigma$-regular conjugacy class of $G$ is infinite.
It is known \cite{Kle, Pa, Om} that $(G,\sigma)$ satisfies Kleppner's condition if and only if $W^*(G,\sigma)$ is a factor,
if and only if $C^*_r(G,\sigma)$ has trivial center, if and only if $C^*_r(G,\sigma)$ is prime.
It follows easily from these equivalences that Kleppner's condition is necessary
for $(G,\sigma)$ to be C$^*$-simple (resp.\ to have the unique trace property).
On the other hand, if $G$ is amenable, then 
$C_r^*(G)\simeq C^*(G)$ has a $1$-dimensional $^*$-representation \cite{Pat}.
Hence, if $G$ is a nontrivial amenable ICC group, then $(G,1)$ satisfies Kleppner's condition,
but neither is $(G,1)$ C$^*$-simple, nor does it have the unique trace property.

\bigskip The following lemma, which is a slight adaptation of a technical result due to Carey and Moran (\cite[Lemma~4.1]{CM}),
will be important in the proof of Theorem~\ref{K}.

\begin{lemma}\label{CM}
Let $G$ be a group and assume $\psi$ is a tracial state on $C^*(G)$.
Let $g \to u(g) \in C^*(G)$ denote the canonical embedding of $G$ into $C^*(G)$ and let $\psi_G$ be the function on $G$ given by $\psi_G= \psi \circ u$.
Assume that there exist $h \in G$ and a sequence $\{g_i\}_{i\in \N}$ in $G$ such that 
\begin{equation}\label{CM-eq}
\psi_G\big(g_jh^{-1} g_j^{-1} g_ih g_i^{-1} \big) = 0
\quad\text{for every } i\neq j \text{ in } \N\,.
\end{equation}
 Then $ \psi_G(h)=0$. 

\end{lemma}
\begin{proof} 

\medskip For each $N \in \N$, let $a_N \in C^*(G)$ be defined by $a_N = I - \overline{\psi_G(h)}\, \sum_{i=1}^N u\big(g_ih g_i^{-1} \big)\,.$
Then we have
$$(a_N)^*\, a_N = \Big[I - \psi_G(h)\, \sum_{j=1}^N u\big(g_jh^{-1} g_j^{-1} \big)\Big] \, \Big[I - \overline{\psi_G(h)}\, \sum_{i=1}^N u\big(g_ih g_i^{-1} \big)\Big]$$
$$= I - \overline{\psi_G(h)}\, \sum_{i=1}^N u\big(g_ih g_i^{-1} \big) - \psi_G(h)\, \sum_{j=1}^N u\big(g_jh^{-1} g_j^{-1} \big)
+ |\psi_G(h)|^2\, \sum_{i,j=1}^N u\big( g_jh^{-1} g_j^{-1} g_ih g_i ^{-1} \big)\,.
$$
Using that $\psi$ is a tracial state, we get
$$\psi\big((a_N)^*\, a_N\big) = 1 - 2N\, |\psi_G(h)|^2 + |\psi_G(h)|^2\, \sum_{i,j=1}^N \psi_G\big(g_j h^{-1} g_j^{-1} g_ih g_i ^{-1} \big)$$
$$= 1 - N \,|\psi_G(h)|^2 +\sum_{i,j=1,\, i\neq j}^N \psi_G\big( g_jh^{-1} g_j^{-1} g_ih g_i^{-1} \big)\,.$$
Using \eqref{CM-eq}, we get
$$ 0\leq \psi\big((a_N)^*\, a_N\big) = 1 - N \,|\psi_G(h)|^2\,,$$
hence $|\psi_G(h)| \leq \sqrt{1/N}$. Letting $N\to \infty$, we obtain the desired conclusion.
\end{proof}

\subsection{On the QTS property}
Let $A$ denote a unital C$^*$-algebra.
Following G.~Murphy \cite{Mur}, $A$ is said to have the \emph{QTS property} if,
for each proper (closed two-sided) ideal $J$ of $A$, the quotient $A/J$ admits a tracial state. 
As observed by Murphy, if $A$ has the QTS property, then $A$ is simple if and only if all its tracial states are faithful.
As an immediate consequence, we get:

\begin{theorem}\label{QTS}
Assume that $A$ has the QTS property and a unique tracial state, which is \emph{faithful}. Then $A$ is simple.
\end{theorem}

We recall (cf.\ \cite{HZ}) that a unital C$^*$-algebra is simple with at most one tracial state if and only if it has the Dixmier property.
Hence, the assumptions of Theorem~\ref{QTS} imply that $A$ has the Dixmier property.
In this connection, we remark that N.~Ozawa has recently shown \cite{Oz} that the QTS property may be characterized by a weaker Dixmier type property.  

\medskip The following result will be useful to us: 

 \begin{corollary}\label{GUTS}
Assume that $G$ is amenable and let $\sigma \in Z^2(G,\T)$.
Then $(G,\sigma)$ is C$^*$-simple whenever it has the unique trace property. 
\end{corollary}

\begin{proof} 
It is known (cf.\ \cite{Mur}) that a unital C$^*$-algebra $A$ has the QTS property whenever $A$ is hypertracial (as defined in \cite{Bed3}).
Since $G$ is amenable if and only if $C^*_r(G,\sigma)$ is hypertracial (cf.\ \cite{Bed3}), the assertion follows from Theorem~\ref{QTS}. 
For the ease of the reader, we  sketch a direct proof that $C_r^*(G,\sigma)$ has the QTS property when $G$ is assumed to be amenable. 
Let $J$ be a proper ideal of $C^*_r(G,\sigma)$,
let $\pi$ denote the canonical quotient map from $C^*_r(G,\sigma)$ onto $B= C^*_r(G,\sigma)/J$,
let $\varphi$ denote a state on $B$, and set $v(g) = \pi(\lambda_\sigma(g))$ for each $g\in G$.
For each $x \in B$, define $x_\varphi\in \ell^\infty(G)$ by
\[
x_\varphi(g) = \varphi\big(v(g)\,x\,v(g)^*\big)\quad\text{for each } g\in G\,.
\]
Now, let $m$ be a right-invariant mean on $\ell^\infty(G)$ and define $\psi\colon B\to \C$ by
\[
\psi(x) = m(x_\varphi)\quad\text{for each }x \in B\,.
\]
Then $\psi$ is a state on $B$.
Moreover, as $\big(v(h)x\,v(h)^*\big)_\varphi$ is the right-translate of $x_\varphi$ by $h$ for each $h\in G$,
the invariance of $m$ gives that $\psi\big(v(h)\,xv\,(h)^*\big) = \psi(x)$ for all $h \in G$ and $x \in B$.
As $\{ v(h)\mid h \in G\}$ generates $B$ as a C$^*$-algebra, it follows readily that $\psi$ is tracial.
\end{proof}

Murphy also shows that a unital C$^*$-algebra $A$ has the QTS property whenever $A$ is exact \cite{BrOz} and has stable rank one
(i.e., the invertible elements of $A$ are dense in $A$).
Now, it follows from \cite{Ex} that $C^*_r(G,\sigma)$ is exact whenever $G$ is exact. Hence, another corollary of Theorem~\ref{QTS} is:

\begin{corollary}\label{GUTS2}
Let $\sigma \in Z^2(G,\T)$. Assume that $G$ is exact and $C_r^*(G, \sigma)$ has stable rank one.
Then $(G,\sigma)$ is C$^*$-simple whenever it has the unique trace property.
\end{corollary}
Not much seems to be known about conditions ensuring that  $C_r^*(G, \sigma)$ has stable rank one. We plan to investigate this in a separate paper.

\subsection{On FC-hypercentral groups}\label{FCH-section}

Let $G$ be a group. We recall that the \emph{FC-center of $G$} is given by
\[
FC(G)=\{ g\in G\mid \text{the conjugacy class of } g \text{ is finite}\}\,.
\]
The FC-center of $G$ is a normal subgroup of $G$, which is trivial if and only if $G$ is ICC. 
The group $G$ is said to be an \emph{FC-group} when $FC(G)=G$.

The \emph{upper (ascending) FC-central series} $\{F_\alpha\}_\alpha$ of $G$ is a normal series of subgroups of $G$ indexed by the ordinal numbers.
It is defined as follows (cf.\ \cite[Section~4.3]{Rob}):

\smallskip We set $F_0=\{e\}$, $F_\alpha/F_\beta =FC(G/F_\beta)$ if $\alpha=\beta+1$,
and $F_\alpha=\bigcup_{\beta<\alpha}F_\beta$ when $\alpha$ is a limit ordinal number.
This series eventually stabilizes and $FCH(G)=\lim_\alpha F_\alpha=\bigcup_\alpha F_\alpha$ is called the \emph{FC-hypercenter of} $G$.
Note that $FCH(G)$ is a normal subgroup of $G$ since it is a union of normal subgroups.

Since $F_1=FC(G)$, we have that $FCH(G)$ is trivial if and only if $FC(G)$ is trivial if and only if $G$ is ICC.

\medskip A group $G$ is called \emph{FC-hypercentral} \cite{Rob} when $FCH(G)=G$.
If the upper FC-central series  stabilizes to $G$ after a finite number of steps, then $G$ is called \emph{FC-nilpotent}.
For example, $G$ is FC-nilpotent whenever $G$ is virtually nilpotent (i.e., it contains a nilpotent subgroup of finite index). If $G$ is  finitely generated and FC-hypercentral, then $G$ is FC-nilpotent; further, if $G$ is finitely generated and FC-nilpotent, then $G$ is a finite extension of finitely generated nilpotent subgroup. (See \cite[Theorem 2 and its proof]{ML}).

\medskip As observed by S.~Echterhoff in \cite{Ech}, it follows that FC-hypercentral groups have polynomial growth and are therefore amenable. Moreover, a deep result proved by Echterhoff is that $G$ is FC-hypercentral if and only if $G$ is amenable and every prime ideal of $C^*(G)$ is maximal.
This generalizes an earlier result of Moore and Rosenberg \cite{MR}, which says that any countable amenable T$_1$-group is FC-hypercentral.
(We recall that $G$ is called a \emph{T$_1$-group}
when every primitive ideal of $C^*(G)$ is maximal).

\medskip 
To sum up, consider the following conditions for a group $G$:
\begin{itemize}\itemsep0pt
\item[(i)] $G$ is virtually nilpotent. 
\item[(ii)] $G$ is FC-nilpotent.
\item[(iii)] $G$ is FC-hypercentral.
\item[(iv)] $G$ is an amenable T$_1$-group.
\item[(v)] $G$ has polynomial growth.
\end{itemize}
In general, we have \,(i) $\Rightarrow$ (ii) $\Rightarrow$ (iii) $\Rightarrow$ (iv), and (iii) $\Rightarrow$ (v). 
If $G$ is countable, then (iii) $\Leftrightarrow$ (iv).
(To our knowledge, it is open whether (iv) implies (iii) in general;
it seems also to be unknown whether all T$_1$-groups have polynomial growth.)
Finally, for a finitely generated group, all the five conditions are equivalent,
the implication (v) $\Rightarrow$ (i) being a famous result due to M.~Gromov \cite{G}.

\medskip Another condition one might consider is:
\begin{itemize}
\item[(vi)] $G$ is elementary amenable with subexponential growth.
\end{itemize}
Then we have (v) $\Rightarrow$ (vi). Indeed, Gromov's result gives 
that any finitely generated subgroup of a group with polynomial growth must be virtually nilpotent,
and thus elementary amenable. Since a group is a direct limit of its finitely generated subgroups,
this means that a group with polynomial growth is elementary amenable. 

If $G$ is a finitely generated group, we also have (vi) $\Rightarrow$ (v), hence conditions (i)-(vi) are all equivalent in this case. This assertion is an immediate consequence of a result due to C. Chou \cite{Cho}, which says that a finitely generated elementary amenable group have either polynomial growth or exponential growth. 

\medskip In \cite{Jaw}, W.~Jaworski defines a group $G$ to be \emph{identity excluding} if the only irreducible unitary representation of $G$
which weakly contains the $1$-dimensional identity representation is the $1$-dimensional identity representation itself.
An interesting result in our context is \cite[Theorem~4.5]{Jaw},
which says that a countable group is FC-hypercentral if and only if it is amenable and identity excluding.


\medskip The FC-hypercenter of of a group $G$ may be described as the smallest normal subgroup of $G$ that produces an ICC quotient group.
This fact is mentioned without proof in \cite[Remark~4.1]{Jaw}. For completeness, we give a proof of this useful characterization.

\begin{proposition}\label{ICC-quotient}
Let $G$ be a group. Then the quotient group $G/FCH(G)$ is ICC.
Moreover, if $N$ is a normal subgroup of $G$ such that $G/N$ is ICC, then $FCH(G) \subset N$.
\end{proposition}

\begin{proof}
Since the upper FC-central series of $G$ stabilizes at $FCH(G)$, the first assertion is clear.
The second assertion also follows from the construction. Assume that $N$ is a normal subgroup of $G$ such that $G/N$ is ICC.
Then the quotient map $G \to G/N$ sends $FC(G)$ to a subgroup of $FC(G/N)=\{e\}$, so one has $F_1 = FC(G) \subset N$.

Next, suppose that $\alpha$ and $\beta$ are ordinals such that $\alpha=\beta+1$ and $F_\beta \subset N$.
Then the quotient map $G/F_\beta \to (G/F_\beta)/(N/F_\beta)=G/N$ sends $FC(G/F_\beta)=F_\alpha/F_\beta$ to the identity, that is, $F_\alpha \subset N$.

Finally, if $\alpha$ is a limit ordinal and $F_\beta \subset N$ for all $\beta < \alpha$,
then $F_\alpha = \bigcup_{\beta<\alpha}F_\beta \subset N$.
Hence, $F_\alpha \subset N$ for all ordinals $\alpha$, so $FCH(G) = \bigcup_\alpha F_\alpha \subset N$.
\end{proof}

\begin{corollary}
A group $G$ is FC-hypercentral if and only if $G$ has no nontrivial ICC quotients,
that is, if and only if $FC(G/N)$ is nontrivial for every proper normal subgroup $N$ of $G$.
\end{corollary}

We also mention some permanence properties of FC-hypercentrality.

\medskip We will say that $G$ is an \emph{FC-central extension} of a group $K$ if there is a short exact sequence
\[
e \longrightarrow H \longrightarrow G \longrightarrow K \longrightarrow e
\]
such that $H \subset FC(G)$.
In particular, $H$ must be an FC-group.
Note that FC-central extensions generalize both central and finite extensions.
The class of FC-nilpotent groups forms the smallest class that is closed under FC-central extensions.

\medskip Similarly, we will say that $G$ is an \emph{FC-hypercentral extension} of $K$ if $H \subset FCH(G)$.
The class of FC-hypercentral groups is closed under FC-hypercentral extensions:

\begin{proposition}\label{FC-central-extension}
Suppose a group $G$ is an FC-hypercentral extension of a group $K$.
Then $G$ is FC-hypercentral if and only if $K$ is FC-hypercentral.
\end{proposition}

\begin{proof}
Let $q \colon G \to K$ denote the canonical surjection.
According to Lemma~\ref{ICC-quotient}, we have to show that $G$ has a nontrivial ICC quotient if and only if $K$ has a nontrivial ICC quotient.

First, let $N$ be a proper normal subgroup of $K$ such that $K/N$ is ICC.
Then $q^{-1}(N)$ is a proper normal subgroup of $G$ that contains $H$ and $G/q^{-1}(N) \cong (G/H)/(q^{-1}(N)/H) \cong K/N$ is ICC.

For the converse implication, let $N$ be a proper normal subgroup of $G$ such that $G/N$ is ICC.
Then $N$ must contain $FCH(G)$, in particular, $N$ must contain $H$.
Hence, $q(N) \cong N/H$ is a proper normal subgroup of $K$ and $G/N \cong (G/H)/(N/H) \cong K/q(N)$, which is ICC.
\end{proof}

It follows from Proposition~\ref{FC-central-extension} that the class of FC-hypercentral groups is closed under quotients and direct products.
We also have:

\begin{proposition}\label{FCH-subgroup}
Suppose $H$ is a subgroup of $G$.
\begin{itemize}\itemsep0pt \parsep0pt
\item[a)] If $G$ is FC-hypercentral, then $H$ is FC-hypercentral.
\item[b)] If $H$ is FC-hypercentral and of finite index in $G$, then $G$ is FC-hypercentral.
\end{itemize}
\end{proposition}
\begin{proof}
For $a)$ we have that $FC(G) \cap H \subset FC(H)$ for any $G$,
and a routine induction argument shows that $FC_\alpha(G) \cap H \subset FC_\alpha(H)$ for all ordinals.
Hence, $FCH(H) \cap H \subset FCH(H)$, so if $G$ is FC-hypercentral, then so is $H$.

\medskip To prove $b)$, assume that $G$ is not FC-hypercentral.
Then there is a proper normal subgroup $N$ of $G$ such that $G/N$ is ICC.
Another routine argument gives that $H \cap N$ is a proper normal subgroup of $H$ and $H/(H \cap N)$ is ICC.
Hence, $H$ is not FC-hypercentral.
\end{proof}

\section{On \texorpdfstring{C$^*$}{C*}-simplicity and the unique trace property for FC-hypercentral groups} 

This section is mainly devoted to the proof of the following result:

\begin{theorem}\label{K} Assume that $G$ is an FC-hypercentral group and let $\sigma \in Z^2(G,\T)$. 

\medskip \noindent Then the following properties are equivalent: 
\begin{itemize}
\item[(i)] $(G,\sigma)$ satisfies Kleppner's condition;
\item[(ii)] $(G,\sigma)$ is C$^*$-simple;
\item[(iii)] $(G,\sigma)$ has the unique trace property,
\end{itemize}

\noindent This means the class of FC-hypercentral groups is contained in the class $\K^{am}$.
\end{theorem}

To simplify notation, we set $A=C_r^*(G,\sigma)$ and let $A_{FC}$ denote the C$^*$-subalgebra of $A$ generated by $\{\lambda_\sigma(g)\mid g\in FC(G)\}$.  

\smallskip A simple computation gives that for all $g,h \in G$, we have
\[
\lambda_\sigma(g)\,\lambda_\sigma(h)\,\lambda_\sigma(g) ^* = \widetilde{\sigma}(g,h)\, \lambda_\sigma(ghg^{-1})\,,
\]
where
\[
\widetilde{\sigma}(g,h) = \sigma(g,h)\, \overline{\sigma(ghg^{-1},g)}\,.
\] 

\medskip Hence, we get an action $\gamma$ of $G$ on $A_{FC}$, given by
\[
\gamma_g(a) = \lambda_\sigma(g) \, a \, \lambda_\sigma(g) ^*\, \quad\text{for all } g\in G \text{ and } a \in A_{FC}\,.
\]
We will let $\tau_{FC}$ denote the tracial state on $A_{FC}$ obtained by restricting the canonical tracial state $\tau$ on $A$ to $A_{FC}$.
Clearly,  $\tau_{FC}$ is $\gamma$-invariant.

\begin{proposition} \label{gamma} The following conditions are equivalent:
\begin{itemize}
\item[(i)] $(G,\sigma)$ satisfies Kleppner's condition,
\item[(ii)] $\tau_{FC}$ is the unique $\gamma$-invariant tracial state of $A_{FC}$.
\end{itemize}

\end{proposition}
\begin{proof}
 $(i)\Rightarrow (ii)$: Assume that $(i)$ holds.
If $G$ is ICC, then $A_{FC}\simeq \mathbb{C}$, and $(ii)$ is trivially satisfied in this case.
We may therefore assume that $G$ is not ICC, so $FC(G)\neq \{e\}$.
Let $\varphi$ be a $\gamma$-invariant state of $A_{FC}$.

\medskip Consider $h\in FC(G), \, h\neq e$.
As $h$ is not $\sigma$-regular, there exists $g\in G$ such that $gh=hg$ and $\sigma(g,h) \neq \sigma(h,g)$.
It clearly follows that $\widetilde{\sigma}(g,h)\neq 1$ and 
\[
\gamma_g\big(\lambda_\sigma(h)\big) = \widetilde{\sigma}(g,h)\, \lambda_\sigma(ghg^{-1}) = \widetilde{\sigma}(g,h)\, \lambda_\sigma(h)\,.
\]
Thus, we get
\[
\varphi\big(\lambda_\sigma(h)\big) = \varphi\big(\gamma_g(\lambda_\sigma(h))\big) = \widetilde{\sigma}(g,h)\, \varphi\big(\lambda_\sigma(h)\big)\,,
\]
so $\varphi\big(\lambda_\sigma(h)\big) = 0$. This implies that $\varphi$ agrees with $\tau_{FC}$. 


\medskip  $(ii)\Rightarrow (i)$: Assume that $(i)$ does not hold.
It is then known that the center $Z$ of $A$ is nontrivial.
In fact, $Z_{FC} := A_{FC} \cap Z$ is nontrivial in this case (cf.\ \cite[Proof of Theorem 2.7]{Om}).
So we may pick a non-scalar positive element $z_0\in Z_{FC}$ and define a tracial state on $A_{FC}$ by 
\[
\varphi(a) = \frac{1}{\tau_{FC}(z_0)} \,\, \tau_{FC}(az_0)\,, \quad \text{for } a \in A_{FC}\,.
\]
As $Z_{FC}$ is fixed by $\gamma$ and $\tau_{FC}$ is $\gamma$-invariant, $\varphi$ is also $\gamma$-invariant.

Now, observe that $\varphi(z_0) \neq \tau_{FC}(z_0)$.
Indeed, assume (for contradiction) that this is not true.
This means that $\tau_{FC}(z_0^{\ 2}) = \tau_{FC}(z_0)^2$, hence
\[
\tau_{FC}\big(\big(z_0-\tau_{FC}(z_0)I\big)^2\big) = 0\,.
\]
Since $\tau_{FC}$ is faithful, we get that $z_0 = \tau_{FC}(z_0) I $, i.e., $z_0$ is a scalar, which gives a contradiction.

This shows  that $\varphi \neq \tau_{FC}$, so  $\tau_{FC}$ is not the only  $\gamma$-invariant tracial state of $A_{FC}$. Thus, $(ii)$ does not hold.  
\end{proof}

\medskip An immediate consequence of Proposition~\ref{gamma} is the following: 

\begin{corollary} \label{Kle-FC}
Assume that $(G,\sigma)$ satisfies Kleppner's condition and let $\varphi$ be a tracial state of $A$. Then $\varphi$ agrees with $\tau$ on $A_{FC}$.
\end{corollary}

Let now $D^\sigma$ denote the subgroup of $\T$ generated by the image of $\sigma$,
i.e., by the set $\sigma(G \times G)$. 
We consider $D^\sigma$ as a discrete group and 
define the extension $G^\sigma$ of $G$ by $D^\sigma$ as the set $G \times D^\sigma$
equipped with the product given by  $(g,z)(h,w)=(gh,\sigma(g,h)zw)$.
Moreover, we let $(g,z) \to u_\sigma(g,z) \in C^*(G^\sigma)$ denote the canonical embedding of  $G^\sigma$ into $C^*(G^\sigma)$.

\begin{lemma}\label{character}
Let $\psi$ be a tracial state on $C^*(G^\sigma)$ and let $\widetilde\psi$ be the function on $G^\sigma$ given by $\widetilde\psi =  \psi\circ u_\sigma$.
Assume that    
\[
\widetilde\psi(g,z)=z\, \widetilde\psi(g,1)
\]
for all $g \in G$ and $z \in D^\sigma$, and
\[
\widetilde\psi(h,z)=0
\]
for all $h \in FC(G)\setminus\{e\}$ and $z \in D^\sigma$.

\medskip Then  \, $\widetilde\psi(h,z)=0$ for all $h \in FCH(G)\setminus\{e\}$ and $z \in D^\sigma$.
\end{lemma}

\begin{proof}
If $G$ is ICC, then $FCH(G)=FC(G)=\{e\}$, so there is nothing to show.
Thus, suppose $G$ is not ICC.

Let $\{F_\alpha\}_\alpha$ denote the upper  FC-central series of $G$ and let $\alpha$ be an ordinal. It suffices to show that 
$\widetilde\psi(h,z)=0$  for all $h \in F_\alpha \setminus\{e\}$ and all $z \in D^\sigma$.

For $\alpha=1$ the result holds by assumption since $F_1=FC(G)$.
So let $\alpha > 1$ be an ordinal and suppose that $\widetilde\psi=0$ on $(F_\beta\setminus\{e\}) \times D^\sigma \subset G^\sigma$ 
for all $\beta < \alpha$.

If $\alpha$ is a limit ordinal, then by construction
\[
\big(F_\alpha\setminus\{e\}\big) \times D^\sigma
=\Big(\big(\bigcup_{\beta<\alpha}F_\beta\big)\setminus\{e\}\Big) \times D^\sigma
=\Big(\bigcup_{\beta<\alpha}\big(F_\beta\setminus\{e\}\big)\Big) \times D^\sigma.
\]
By hypothesis, $\widetilde\psi=0$ on the right-hand side, hence also on the left-hand side.

\medskip If $\alpha$ is a successor ordinal, then $\alpha=\beta+1$ for some ordinal $\beta$.
Pick an element $h \in F_\alpha \setminus F_\beta$.
Then the set $\{ ghg^{-1} \mid g \in G \}$ is infinite since $FC(G) \subset F_\beta$,
while the set $\{ ghg^{-1}F_\beta \mid g \in G \}$ is finite in $G/F_\beta$ since $F_\alpha/F_\beta=FC(G/F_\beta)$.
Hence, there exists an infinite sequence $(g_i)_{i\in \N}$ in $G$ such that $g_ihg_i^{-1}F_\beta=g_jhg_j^{-1}F_\beta$ for all $i,j \in \N$,
and $g_ihg_i^{-1} \neq g_jhg_j^{-1}$ whenever $i \neq j$.
This means that $g_jh^{-1}g_j^{-1}g_ihg_i^{-1}\in F_\beta\setminus\{e\}$ whenever $i \neq j$ in $\N$.
Since
\[
(g_j,1)(h,1)^{-1}(g_j,1)^{-1}(g_i,1)(h,1)(g_i,1)^{-1} = (g_jh^{-1}g_j^{-1}g_ihg_i^{-1}g_j,w)
\]
for some $w \in D^\sigma$, and $\widetilde\psi=0$ on $(F_\beta\setminus\{e\}) \times D^\sigma$, we get
\[
\widetilde\psi\Big((g_j,1)(h,1)^{-1}(g_j,1)^{-1}(g_i,1)(h,1)(g_i,1)^{-1}\Big)=\widetilde\psi(g_jh^{-1}g_j^{-1}g_ihg_i^{-1},w)=0
\]
whenever $i \neq j$ in $\N$. 
We may now apply Lemma~\ref{CM} and conclude that $\widetilde\psi(h,1)=0$.
Thus, $\widetilde\psi(h,z)= z\, \widetilde\psi(h,1) = 0$ for all $h \in F_\alpha\setminus\{e\}$ and $z \in D^\sigma$.

\end{proof}
Let $A_{FCH}$ denote the C$^*$-subalgebra of $A$ generated by $\{ \lambda_\sigma(h) \mid h \in FCH(G) \}$.
\begin{lemma}\label{FCH-trace}
Let $\varphi$ be a tracial state on $A$ which agrees with $\tau$ on $A_{FC}$.
Then $\varphi$ agrees with $\tau$ on $A_{FCH}$.
\end{lemma}

\begin{proof}
As the map $(g,z) \mapsto z\lambda_\sigma(g)$ is a unitary representation of $G^\sigma$ on $\ell^2(G)$,
there exists a surjective $^*$-homomorphism $\pi\colon C^*(G^\sigma) \to A$ satisfying $\pi(u_\sigma(g,z))=z\,\lambda_\sigma(g)$ for each $(g,z) \in G^\sigma$.
Thus $\varphi$ lifts to a tracial state $\psi = \varphi \circ \pi$ on $C^*(G^\sigma)$.
For all $g \in G$ and $z \in D^\sigma$, we have
\[
\psi(u_\sigma(g,z))=
\varphi(z\,\lambda_\sigma(g))=z\,\varphi(\lambda_\sigma(g))=
z\,\psi(u_\sigma(g,1)).
\]
Since $\varphi(\lambda_\sigma(g)) = 0$ \, for all $g \in FC(G) \setminus\{e\}$ (by assumption),
it follows that $\psi(u_\sigma(g,z))= 0$ for all $g \in FC(G)\setminus\{e\}$ and $z \in D^\sigma$.
Hence, using Lemma~\ref{character}, we get that
$\psi(u_\sigma(g,z))=0$ for all $g \in FCH(G) \setminus\{e\}$ and $z \in D^\sigma$.
In particular, we get
\[
\varphi(\lambda_\sigma(g))= \psi(u_\sigma(g,1))=0
\]
for all $g \in FCH(G) \setminus\{e\}$.
Thus, $\varphi$ agrees with $\tau$ on $A_{FCH}$, as desired.
\end{proof}

\begin{proof}[Proof of Theorem~\ref{K}]
Assume that $(i)$ holds and let $\varphi$ be a tracial state on $A$. Corollary~\ref{Kle-FC} tells us that $\varphi$ agrees with $\tau$ on $A_{FC}$.
Since $G$ is FC-hypercentral, we have that $A=A_{FCH}$. Applying Lemma~\ref{FCH-trace}, we conclude that $\varphi$ coincides with $\tau$.
Hence, $(iii)$ holds. 
Since FC-hypercentral groups are amenable, we get from Corollary~\ref{GUTS} that $(ii)$ follows from $(iii)$.
Finally, we know that $(ii)$ implies $(i)$ in general.
\end{proof}

We will also show how Echterhoff's characterization of FC-hypercentrality mentioned in Section~\ref{FCH-section}
may be used to give a different proof of the implication $(i) \Rightarrow (ii)$ in Theorem~\ref{K}. We will need the following result:

\begin{proposition}\label{prime-maximal}
Assume that every prime ideal of $C^*(G^\sigma)$ is maximal.
Then $(G,\sigma)$ is C$^*$-simple whenever $(G,\sigma)$ satisfies Kleppner's condition.
\end{proposition}

\begin{proof}
Assume that $(G,\sigma)$ satisfies Kleppner's condition. Then $C_r^*(G,\sigma)$ is prime \cite[Theorem~2.7]{Om}.
Let $\pi$ denote the surjective $^*$-homomorphism $ C^*(G^\sigma) \to C_r^*(G,\sigma)$ obtained in the proof of Lemma~\ref{FCH-trace}.  
The kernel $\mathcal{J}$ of $\pi$ is then a prime ideal of $C^*(G^\sigma)$.
Hence, the assumption gives that $\mathcal{J}$ is maximal, so $C^*(G^\sigma)/\mathcal{J} \simeq C_r^*(G,\sigma)$ is simple.
\end{proof}
 
\begin{proof}[Another proof of $(i) \Rightarrow (ii)$ in Theorem~\ref{K}]
Using Proposition~\ref{FC-central-extension}, we get that $G^\sigma$ is also FC-hypercentral.
Hence, \cite[Corollary~3.2]{Ech} tells us that every prime ideal of $C^*_r(G^\sigma)$ is maximal,
so the assertion follows from Proposition~\ref{prime-maximal}.
\end{proof}

We also record another consequence of Proposition~\ref{prime-maximal}. 
\begin{corollary}\label{T1}
Assume that $G$ is a countable discrete group such that $G^\sigma$ is a T$_1$-group.
Then $(G,\sigma)$ is C$^*$-simple whenever $(G,\sigma)$ satisfies Kleppner's condition.
\end{corollary}
\begin{proof}
Since $G$ is countable, $G^\sigma$ is also countable.
Hence, $C^*(G^\sigma)$ is separable, so its prime ideals are primitive ideals,
and the assumption in Proposition~\ref{prime-maximal} is therefore satisfied.
\end{proof}

We note that if $G$ is amenable and it satisfies the assumptions in Corollary~\ref{T1},
then it is not difficult to deduce from the Moore-Rosenberg result cited in Section~\ref{FCH-section} that $G$ is FC-hypercentral.
Hence, Corollary~\ref{T1} is covered by Theorem~\ref{K} in this case.

\medskip We do not know of any group in $\K^{am}$ that is not FC-hypercentral. Thus, a natural question is:  

\begin{question}
Does the class of FC-hypercentral groups coincide with $\K^{am}$ ? 

\medskip This question may be reformulated as follows. Assume that $G$ is neither ICC, nor FC-hypercentral, but amenable.
Can one always find $\sigma \in Z^2(G,\sigma)$ such that $(G,\sigma)$ satisfies Kleppner's condition,
but $(G,\sigma)$ is not C$^*$-simple or does not have the unique trace property ?
It does not seem easy to answer this question positively.
\end{question}

\section{On the unique trace property for a larger class of groups}
In this section, we let $G$ be a group and $\sigma \in Z^2(G,\T)$.
Motivated by Proposition~\ref{ICC-quotient}, we define $ICC(G)=G/FCH(G)$.
Moreover, we let $\P$ denote the class of  groups considered in \cite{BC2},
that consists of all PH groups  \cite{Pro} and of all groups satisfying property ($P_{\rm com}$) \cite{BCH}.
The class $\P$ is a large subclass of the class of ICC groups, and the only amenable group belonging to $\P$ is the trivial group.
The main purpose of this section is to show the following result:

\begin{theorem}\label{FCH-P}
Assume that $K = ICC(G)$ belongs to $\P$. Then we have:

\begin{itemize}
\item[a)] $(G,\sigma)$ satisfies Kleppner's condition if and only if $(G,\sigma)$ has the unique trace property. 
\item[b)] Set $H= FCH(G)$ and let $\sigma_H$ denote the restriction of $\sigma$ to $H\times H$.
If $(H, \sigma_H)$ satisfies Kleppner's condition, then $(G,\sigma)$ is C$^*$-simple and has the unique trace property.
\end{itemize}

\end{theorem}

We set  $A= C_r^*(G,\sigma)$ and let $A_{FC}$  and $A_{FCH}$ be defined as in the previous section.
Since $FCH(G)$ is normal in $G$, the action $\gamma$ of $G$ on $A_{FC}$ extends to an action $\widetilde\gamma$ of $G$ on $A_{FCH}$ given by
\[
\widetilde\gamma_g(a) = \lambda_\sigma(g) \,a \,\lambda_\sigma(g)^*
\]
for all $g\in G$ and $a \in A_{FCH}$.
Let $\tau_{FCH}$ denote the restriction of the canonical tracial state $\tau$ on $A$ to $A_{FCH}$.
Clearly, $\tau_{FCH}$ is $\widetilde\gamma$-invariant. Analogously to Proposition~\ref{gamma}, we have:

\begin{proposition}\label{FCH-gamma}
The following conditions are equivalent:
\begin{itemize}
\item[(i)] $(G,\sigma)$ satisfies Kleppner's condition,
\item[(ii)] $\tau_{FCH}$ is the unique $\widetilde\gamma$-invariant tracial state of $A_{FCH}$.
\end{itemize}
\end{proposition}

\begin{proof}
$(i) \Rightarrow (ii)$:
Assume that $(i)$ holds.
If $G$ is ICC, then $A_{FCH} \cong \mathbb{C}$, and $(ii)$ is trivially satisfied in this case.
We may therefore assume that $G$ is not ICC, so $FC(G) \neq \{e\}$.
Let $\phi$ be a $\widetilde\gamma$-invariant tracial state of $A_{FCH}$.
As in Proposition~\ref{gamma}, we compute that $\phi(\lambda_\sigma(h))=0$ for all $h \in FC(G) \setminus \{e\}$.
Now, let $E$ denote the canonical conditional expectation from $A$ onto $A_{FCH}$.
Then $\varphi := \phi \circ E$ is a tracial state on $A$ such that $\varphi(\lambda_\sigma(h))=0$ for all $h \in FC(G) \setminus \{e\}$.
Applying Lemma~\ref{FCH-trace}, we get that $\varphi = \phi_{|A_{FCH}} =\tau_{FCH}$.

Since $A_{FC} \subset A_{FCH}$,
the proof of $(ii) \Rightarrow (i)$ goes along the same lines as the one given for this implication in Proposition~\ref{gamma}.
\end{proof}

Set $H=FCH(G)$ and $K=ICC(G) = G/H$,
and let $q$ denote the canonical homomorphism from $G$ onto $K$. 
Further, let $n\colon K\to G$ be a section for $q$ satisfying $n(e)=e$,
and define $m\colon K\times K \to H$ by $m(k,l)=n(k)n(l)n(kl)^{-1}$.
Finally, let $\sigma_H$ denote the restriction of $\sigma$ to $H\times H$. 

Moreover, define $\beta\colon K\to \aut (A_{FCH})$ by $\beta = \widetilde\gamma\circ n$ and $\omega\colon K \times K\to {\mathcal U}(A_{FCH})$ by
\[
\omega(k,l) = \sigma\big(n(k),n(l)\big)\, \sigma\big(m(k,l),n(kl)\big)^*\, \lambda_{\sigma}(m(k,l))\,.
\]
Then $(\beta,\omega)$ is a twisted action of $K$ on $A_{FCH}$
such that $C^*_r(G,\sigma)$ is $^*$-isomorphic to the twisted crossed product $ C^*_r(A_{FCH},K,\beta,\omega)$.
This follows from \cite{Bed} after noticing that $A_{FCH}$ may be identified with $C^*_r(H, \sigma_H)$
via the $^*$-isomorphism sending $\lambda_\sigma(h)$ to $\lambda_{\sigma_H}(h)$ for each $h \in H$. 

\medskip Since $\tau_{FCH}$ is $\widetilde\gamma$-invariant, $\tau_{FCH}$ is also $\beta$-invariant. Moreover, we have:

\begin{proposition}\label{beta} 
The following conditions are equivalent:
\begin{itemize}
\item[(i)] $(G,\sigma)$ satisfies Kleppner's condition,
\item[(ii)] $\tau_{FCH}$ is the unique $\beta$-invariant tracial state of $A_{FCH}$.
\end{itemize}
\end{proposition}

\begin{proof}
Assume $(i)$ holds.
Proposition~\ref{FCH-gamma} gives that $\tau_{FCH}$ is the unique $\widetilde\gamma$-invariant tracial state on $A_{FCH}$.
Consider now a $\beta$-invariant tracial state $\omega$ of $A_{FCH}$ and let $g\in G$.
Write $g = h \, n(k)$ where $k = q(g) \in K$ and $ h = g \, n(k)^{-1} \in H$.
Then, for each $s \in H$, we have
$$ \omega\big(\widetilde\gamma_g(\lambda_{\sigma}(s) ) \big) =  \omega\big(\widetilde\gamma_h\beta_k(\lambda_{\sigma}(s) ) \big) $$
$$ = \omega\big(\lambda_{\sigma}(h) \,\beta_k(\lambda_{\sigma}(s) ) \,\lambda_{\sigma}(h)^* \big) = 
\omega\big(\beta_k(\lambda_{\sigma}(s) ) \big) = \omega (\lambda_{\sigma}(s) )\,.$$
It follows that $\omega$ is $\widetilde\gamma$-invariant.
Hence, $\omega= \tau_{FCH}$.
This shows that $(ii)$ holds.

Conversely, if $(ii)$ holds, then, as $\beta = \widetilde\gamma\circ n$,
it is clear that $\tau_{FCH}$ is the unique $\widetilde\gamma$-invariant tracial state of $A_{FCH}$,
so $(i)$ holds by using Proposition~\ref{FCH-gamma}.
\end{proof}

\begin{proof}[Proof of Theorem~\ref{FCH-P}]

$a)$ From \cite[Corollary~3.9]{BC2}, we know that when $K$ belongs to the class $\P$,
the tracial states of  $C_r^*(A_{FCH}, K, \beta, \omega)$
are in a one-to-one correspondence with the $\beta$-invariant tracial states of $A_{FCH}$.
Hence, it follows from Proposition~\ref{beta} that $C^*_r(G,\sigma) \cong C^*_r(A_{FCH},K,\beta,\omega)$ has a unique tracial state
if and only if $(G,\sigma)$ satisfies Kleppner's condition. 

$b)$ Assume that $(H, \sigma_H)$ satisfies Kleppner's condition. Since $H$ is FC-hypercentral,
we get from Theorem~\ref{K} that $A_{FCH} \simeq C_r^*(H,\sigma_H)$ is simple with a unique tracial state.
This implies that $A_{FCH}$ has a unique $\beta$-invariant tracial state and that the system $(A_{FCH}, K, \beta, \omega)$ is minimal.
Hence, it follows from \cite[Corollary~3.11]{BC2} that $C^*_r(G,\sigma) \cong C^*_r(A_{FCH},K,\beta,\omega)$ is simple with a unique tracial state.
\end{proof}

Let $\ICCP$ denote the class of groups satisfying the assumption in Theorem~\ref{FCH-P}.
Part $a)$ of this theorem shows that $\ICCP$ is contained in the class $\K_{UT}$.
We  believe that $\ICCP \subset \K$, i.e., that we also have $\ICCP \subset \K_{C^*S}$,
but we have not been able to prove this. Part $b)$ of Theorem~\ref{FCH-P} is a somewhat weaker statement;
its proof shows that we would have $\ICCP \subset \K_{C^*S}$ if one could answer positively the following:

\begin{question}\label{minimal}
Assume that $(G,\sigma)$ satisfies Kleppner's condition. Is the system $(A_{FCH}, K, \beta, \omega)$ always minimal ?
That is, is $\{ 0\}$ the only proper $\beta$-invariant ideal of $A_{FCH}$ ?
\end{question} 

In this regard, we also remark that if $G$ belongs to $\ICCP$, $G$ is exact, and $C^*_r(G,\sigma)$ has stable rank one whenever $(G,\sigma)$ satisfies Kleppner's condition, then Corollary~\ref{GUTS2} and Theorem~\ref{FCH-P}~a) together give that $G$ belongs to $\K$.

\medskip Note that if $G/FC(G)$ belongs to $\P$, then $G/FC(G)$ is ICC,
so the upper FC-central series of $G$ stops at $F_1$, i.e., $FCH(G)=FC(G)$.
Hence, Theorem~\ref{FCH-P} gives:

\begin{corollary}\label{FCH-P-II}
Assume that $G/FC(G)$ belongs to $\P$. Then we have:

\begin{itemize}
\item[a)]  $(G,\sigma)$ satisfies Kleppner's condition if and only if $(G,\sigma)$ has the unique trace property.
\item[b)] Set $H=FC(G)$ and let $\sigma_H$ denote the restriction of $\sigma$ to $H\times H$.
If $(H, \sigma_H)$ satisfies Kleppner's condition, then $(G,\sigma)$ is C$^*$-simple and has the unique trace property.
\end{itemize}

\end{corollary}

\begin{example}
Let $n \in \N, \, n\geq 2$ and set $G=\langle a,b \mid ab^n=b^na \rangle$.
Then $G$ is a so-called Baumslag-Solitar group, often denoted by $BS(n,n)$.
We have
\[
FCH(G)=FC(G)=Z(G)=\langle b^n \rangle\simeq\, \Z
\]
and $ICC(G)\simeq\Z*\Z_n\in\mathcal P$ (since $\Z*\Z_n$ is a Powers group \cite{H}), so $G \in \ICCP$.

\medskip Let $f$ denote  the surjective homomorphism $f\colon G\to\Z^2$ satisfying $f(a)=(1,0)$ and $f(b)=(0,1)$.
For $\theta\in\T$, let $\omega_\theta \in Z^2(\Z^2, \T)$ be given by $\omega_\theta(m,n)=e^{2\pi i\theta m_2n_1}$,
and define $\sigma_\theta \in Z^2(G,\T)$ by $\sigma_\theta(x,y)=\omega_\theta(f(x),f(y))$.
It can be shown that every two-cocycle on $G$ is cohomologous to one of this form.

\medskip Then one easily verifies that $(G,\sigma_\theta)$ satisfies Kleppner's condition  if and only if $\theta$ is irrational.
Hence, Theorem~\ref{FCH-P}~a) gives that $(G,\sigma_\theta)$ has the unique trace property  if and only if $\theta$ is irrational.
Theorem~\ref{FCH-P}~b) is not useful in this example since $\sigma_\theta$ restricts to $1$ on $Z(G)$.
However, it can be shown
that $(G,\sigma_\theta)$ is C$^*$-simple if and only if $\theta$ is irrational.
Hence, it follows that $G=BS(n,n)$ belongs to $\K$. 
\end{example}

\medskip We will come back to this example and also discuss other conditions
ensuring simplicity and/or uniqueness of the tracial state for reduced twisted group C$^*$-algebras in a subsequent paper.

\medskip \noindent
E.B.{'}s address:
Institute of Mathematics,
University of Oslo,
P.B.\ 1053 Blindern,
0316 Oslo,
Norway.
E-mail : bedos@math.uio.no.

\medskip \noindent 
T.O.{'}s address:
School of Mathematical and Statistical Sciences,
Arizona State University,
P.O.\ Box 871804,
Tempe, AZ 85287-1804,
USA.
E-mail : omland@asu.edu.
\end{document}